\documentclass[11pt,a4paper,reqno]{amsart}

\usepackage[T1]{fontenc}
\usepackage[english]{babel}
\usepackage{mathtools}
\usepackage{amssymb}
\usepackage{libertinus}
\usepackage[cal=boondoxo,bb=ams]{mathalfa}
\usepackage{microtype}
\usepackage{enumitem}
\usepackage{array,booktabs,tabularx}
\usepackage{etoolbox}

\allowdisplaybreaks[2]
\usepackage[
 a4paper,
 left=30mm,
 right=30mm,
 top=27mm,
 bottom=30mm,
 headsep=8mm,
 footskip=13mm,
 heightrounded
]{geometry}

\setlist[itemize]{leftmargin=2.1em,itemsep=0.25em,topsep=0.45em,parsep=0pt}
\setlist[enumerate]{leftmargin=2.1em,itemsep=0.25em,topsep=0.45em,parsep=0pt}

\numberwithin{equation}{section}

\makeatletter
\def\@settitle{%
 \begin{center}
  \vspace*{-0.4em}
  {\normalfont\bfseries\fontsize{17}{21}\selectfont\@title\par}
  \vspace{0.7em}
 \end{center}}
\renewcommand\section{\@startsection{section}{1}{\z@}%
 {1.5\baselineskip plus 0.2\baselineskip minus 0.1\baselineskip}%
 {0.65\baselineskip}{\normalfont\Large\bfseries}}
\renewcommand\subsection{\@startsection{subsection}{2}{\z@}%
 {1.15\baselineskip plus 0.2\baselineskip minus 0.1\baselineskip}%
 {0.45\baselineskip}{\normalfont\large\bfseries}}
\renewcommand\subsubsection{\@startsection{subsubsection}{3}{\z@}%
 {0.9\baselineskip plus 0.15\baselineskip minus 0.1\baselineskip}%
 {0.35\baselineskip}{\normalfont\normalsize\bfseries}}
\makeatother

\makeatletter
\patchcmd{\@setauthors}{\centering\footnotesize}{\centering\normalsize}{}{}
\patchcmd{\@setauthors}{\MakeUppercase{\authors}}{\authors}{}{}
\patchcmd{\maketitle}{\uppercasenonmath\shorttitle}{}{}{}
\patchcmd{\maketitle}{\@nx\MakeUppercase{\the\toks@}}{\the\toks@}{}{}
\makeatother

\usepackage{hyperref}
\hypersetup{hidelinks,pdfencoding=auto}
\usepackage[nameinlink,capitalise,noabbrev]{cleveref}

\theoremstyle{plain}
\newtheorem{theorem}{Theorem}[section]
\newtheorem{lemma}[theorem]{Lemma}
\newtheorem{prop}[theorem]{Proposition}
\newtheorem{coro}[theorem]{Corollary}

\theoremstyle{definition}

\theoremstyle{remark}
\newtheorem{remark}[theorem]{Remark}

\newcommand{\ml}{\mathcal}
\newcommand{\mb}{\mathbb}
\newcommand{\dd}{\,\mathrm{d}}
\newcommand{\lin}{\mathrm{lin}}
\newcommand{\reg}{\mathrm{reg}}

\newcommand{\high}{\mathrm{high}}

\title[Critical exponent for strongly damped waves]{The critical exponent for the two-dimensional semilinear wave equation with strong damping}

\author[W. Chen]{Wenhui Chen}
\address{School of Mathematics and Information Science, Guangzhou University, Guangzhou 510006, P. R. China}
\email{wenhui.chen.math@gmail.com}

\date{}

\begin{document}

\begin{abstract}
	In this manuscript, we determine the critical exponent for the two-dimensional semilinear wave equation with strong damping. A recent result in D'Abbicco (arXiv, 2026) gives global in-time small data solutions for $p>\frac{10}{3}$, whereas we in the paper prove finite-time blow-up for $2<p\leqslant\frac{10}{3}$ under an averaged sign condition on the initial velocity. Together with the previously known blow-up result for $1<p\leqslant3$, the threshold for the power nonlinearity $|u|^p$ is
	\begin{align*}
		p=p_{\mathrm{crit}}=\frac{10}{3}.
	\end{align*}
	We prove positivity of the full two-dimensional velocity fundamental solution and construct a positive averaged kernel, which allows us to establish a nonlinear lower-bound iteration in a parabolic region.
\end{abstract}

\maketitle

\section{Introduction}

This manuscript determines the critical exponent for the following two-dimensional Cauchy problem:
\begin{align}\label{Main-Problem}
	\begin{cases}
		u_{tt}-\Delta u-\Delta u_t=|u|^p,&x\in\mb{R}^2,\ t>0,\\
		(u,u_t)(0,x)=(u_0,u_1)(x),&x\in\mb{R}^2,
	\end{cases}
\end{align}
with the power exponent $p>1$. The first systematic study, to the best of the author's knowledge, of \eqref{Main-Problem} established small data global in-time existence for $p>4$ and finite-time blow-up for $1<p\leqslant3$ in \cite[Theorem~2 and Theorem~9]{DAbbicco-Reissig=2014}. The determination of the critical exponent was explicitly left open in \cite[Remark~10]{DAbbicco-Reissig=2014} and was subsequently reiterated as an open problem in \cite{Kainane-Kainane-Reissig=2020,Jleli-Samet-Vetro=2021,Fino-Hamza=2022,Kirane-Fino-Kerbal-Laadhari=2024,DAbbicco-Lagioia=2025,Chen-Girardi=2025,Duong-Dao-Reissig=2025,DAbbicco=2026}. The three-dimensional case was recently settled in \cite{Chen-StronglyDamped-3D=2026}, whereas the two-dimensional case remained open. More recently, \cite{DAbbicco=2026} proved global in-time small data solutions for $p>\frac{10}{3}$ for a broad class of semilinear wave equations with damped oscillations, including the two-dimensional strongly damped wave equation \eqref{Main-Problem} with $u_0\equiv0$. Thus, the remaining open range is
\begin{align}\label{Open-range}
	3<p\leqslant\frac{10}{3}.
\end{align}
In the present paper, we establish finite-time blow-up for $2<p\leqslant\frac{10}{3}$, which, together with the known result for $1<p\leqslant3$, yields $p_{\mathrm{crit}}=\frac{10}{3}$. Combined with the three-dimensional result in \cite{Chen-StronglyDamped-3D=2026}, the critical exponent for the semilinear strongly damped wave equation can be written as
\begin{align*}
	\widetilde{p}_{\mathrm{crit}}(n)=\frac{n+1}{n-1}+\frac{1}{3}\ \ \mbox{for}\ \ n\in\{2,3\}.
\end{align*}
See also \cite[$\theta_0=2$ in the equation (15)]{DAbbicco=2026}.

The main difficulty comes from the mixed hyperbolic-parabolic character of the strongly damped wave equation. Indeed, the damping term $-\Delta u_t$ destroys the finite propagation property of the classical wave equation, while the low-frequency part retains wave oscillations broadened by diffusion. The three-dimensional argument in \cite{Chen-StronglyDamped-3D=2026} uses the positivity of the velocity fundamental solution and a dimension-descent identity to obtain a positive kernel on the half-line. This exact reduction is special to three dimensions and is no longer available in the present setting for $n=2$, especially, the weaker pointwise decay of the two-dimensional moving front in the nonlinear iteration.

Several aspects of the linear problem associated with \eqref{Main-Problem} have been studied extensively. Optimal growth/decay estimates, diffusion phenomena, and large-time asymptotic profiles were obtained in \cite{Shibata=2000,Ikehata-Todorova-Yordanov=2013,DAbbicco-Ebert=2014,Ikehata=2014,Ikehata-Onodera=2017,Michihisa=2021,Ikehata-Takeda=2026}. More general semilinear problems for structurally damped wave and $\sigma$-evolution equations can be found in \cite{Pham-Kainane-Reissig=2015,DAbbicco-Ebert=2017,Ikehata-Takeda=2017,Dao-Reissig=2019,Kainane-Kainane-Reissig=2020} and the references therein. The critical exponent for strongly non-effective $\sigma$-evolution equations was determined in \cite{DAbbicco-Ebert=2022} away from the wave endpoint $\sigma=1$, while related nonexistence results for the time-dependent coefficients' cases were established in \cite{Jleli-Samet-Vetro=2021,Fino-Hamza=2022,Kirane-Fino-Kerbal-Laadhari=2024}.

We first prove the positivity of the full two-dimensional velocity fundamental solution. Circular averaging then produces a positive kernel $D(t,r,\rho)$ and the lower-bound inequality
\begin{align*}
	w(t,r)\geqslant W_{\lin}(t,r)
	+\int_0^t\int_0^\infty
	D(t-s,r,\rho)\rho^{\frac{1-p}{2}}|w(s,\rho)|^p\dd\rho\dd s,
\end{align*}
where $w(t,r):=r^{\frac12}\ml M[u(t,\cdot)](r)$ is defined in \eqref{Defn-ml-M}. In the subcritical case $2<p<\frac{10}{3}$, a finite iteration in a parabolic characteristic region, followed by a comparison with a scalar ODE in a strict interior region, gives finite-time blow-up. At the critical exponent $p=\frac{10}{3}$, a refined iteration on moving shells exploits the borderline logarithmic gain and leads again to finite-time blow-up.

\medskip
\paragraph{Notation.}
Throughout this paper, the constants $C$ and $c$ are positive and may change from line to line. We write $f\lesssim g$ if $f\leqslant Cg$, $f\gtrsim g$ if $g\lesssim f$, and $f\approx g$ if both $f\lesssim g$ and $g\lesssim f$ hold. A subscript attached to these symbols indicates the parameters on which the implicit constant is allowed to depend. The underlying spatial domain is always $\mb R^2$ whenever no domain is displayed explicitly. The Fourier transform and its inverse are defined, respectively, via
\begin{align*}
	\widehat f(\zeta):=\ml F[f](\zeta)&:=\int_{\mb R^n}\mathrm{e}^{-ix\cdot\zeta}f(x)\dd x,\\
	\ml F^{-1}[g](x)&:=\frac{1}{(2\pi)^n}\int_{\mb R^n}\mathrm{e}^{ix\cdot\zeta}g(\zeta)\dd\zeta.
\end{align*}
For a function $h=h(t)$, its Laplace transform is written as
\begin{align*}
	\widetilde h(s):=\ml L[h](s):=\int_0^\infty\mathrm{e}^{-st}h(t)\dd t.
\end{align*}
We denote by $B_R:=\{x\in\mb R^2:|x|<R\}$ the open ball centered at the origin with radius $R>0$. For a function $f=f(x)$ on $\mb R^2$, its circular mean is defined by
\begin{align}\label{Defn-ml-M}
	\ml M[f](r):=\frac{1}{2\pi}\int_0^{2\pi}f(r\omega_\theta)\dd\theta\ \ \mbox{for}\ \ r\geqslant0,
\end{align}
where $\omega_\theta:=(\cos\theta,\sin\theta)$.

\section{Main results}\label{Section-Main}

Let $S_0(t)$ and $S_1(t)$ denote the position $u|_{t=0}$ and velocity $u_t|_{t=0}$ propagators for the homogeneous equation associated with \eqref{Main-Problem}, respectively. Thus, the solution to the homogeneous problem with initial data $(v_0,v_1)$ is given by $S_0(t)v_0+S_1(t)v_1$.

For $T\in(0,\infty]$, a function
\begin{align*}
	u\in\ml C\bigl([0,T),H^2\bigr)\cap\ml C^1\bigl([0,T),L^2\bigr)
\end{align*}
is called a mild Sobolev solution to \eqref{Main-Problem} on $[0,T)$ provided that
\begin{align*}
	u(t,\cdot)=S_0(t)u_0(\cdot)+S_1(t)u_1(\cdot)+\int_0^tS_1(t-s)|u(s,\cdot)|^p\dd s\ \ \mbox{in}\ \ H^2
\end{align*}
holds for every $t\in[0,T)$. The standard local in-time well-posedness argument yields, for $p>2$ and smooth compactly supported initial data, a unique maximal mild Sobolev solution on $[0,T_{\max})$, where $T_{\max}\in(0,\infty]$. For $p\leqslant \frac{10}{3}$ additionally, we have the next blow-up result.

\begin{theorem}[Finite-time blow-up]\label{Thm-Blowup}
	Let $2<p\leqslant p_{\mathrm{crit}}=\frac{10}{3}$. Assume that $u_0,u_1\in\ml C_0^\infty$ satisfy
	\begin{align}\label{Positive-Mean}
		M_1:=\int_{\mb{R}^2}u_1(x)\dd x>0.
	\end{align}
	Then the maximal mild Sobolev solution to \eqref{Main-Problem} satisfies $T_{\max}<\infty$.
\end{theorem}

\begin{remark}\label{Remark-Lower-Range}
	For $1<p\leqslant3$, finite-time blow-up for suitable nonnegative compactly supported data was established in \cite[Theorem~9]{DAbbicco-Reissig=2014}. The assumptions of that result are preserved under positive rescaling, so finite-time blow-up also occurs for arbitrarily small smooth compactly supported data in this range. Theorem~\ref{Thm-Blowup} closes the remaining interval \eqref{Open-range}.
\end{remark}

The preceding theorem, together with the global in-time existence result in \cite[Example~1 and Theorem~4]{DAbbicco=2026} for $p>\frac{10}{3}$, yields the exact threshold in two dimensions.

\begin{coro}[Critical exponent]\label{Coro-Critical}
	In the class of velocity data $u_0\equiv0$, the critical exponent for \eqref{Main-Problem} is
	\begin{align*}
		p_{\mathrm{crit}}=\frac{10}{3}.
	\end{align*}
	More precisely, global in-time small data solutions exist for $p>p_{\mathrm{crit}}$ in the setting of \cite{DAbbicco=2026}, whereas for every $1<p\leqslant p_{\mathrm{crit}}$ there exist smooth compactly supported data with arbitrarily small size whose corresponding solutions blow up in finite time.
\end{coro}

\section{Positivity, pointwise estimates, and circular reduction}\label{Section-Kernels}

\subsection{Positivity and the circular kernel}

Let $G_1^{(2)}(t,x)$ denote the two-dimensional velocity fundamental solution, initially understood as the tempered distribution defined by
\begin{align*}
	\widehat G_1^{(2)}(t,\zeta)=\widehat S_1(t,\zeta).
\end{align*}
Since its Fourier transform depends only on $|\zeta|$, according to the representation via the modified Bessel functions, the kernel is radial and we write $G_1^{(2)}(t,r)$ for $r=|x|$.

\begin{lemma}[Laplace representation]\label{Lemma-Laplace-Kernel}
	For $s>0$ and $r>0$, one has
	\begin{align}\label{Laplace-G2}
		\widetilde G_1^{(2)}(s,r)=\frac{1}{2\pi(1+s)}K_0\left(\frac{rs}{\sqrt{1+s}}\right),
	\end{align}
	where $K_0$ denotes the modified Bessel function of the second kind  of order zero. Furthermore, for every $r>0$, the right-hand side of \eqref{Laplace-G2} is completely monotone as a function of $s$.
\end{lemma}

\begin{proof}
	The Laplace-transformed multiplier is the same as in the three-dimensional problem, namely,
	\begin{align*}
		\widetilde{\widehat G}_1^{(2)}(s,\zeta)=\frac{1}{s^2+s|\zeta|^2+|\zeta|^2}=\frac{1}{1+s}\frac{1}{|\zeta|^2+\frac{s^2}{1+s}}.
	\end{align*}
	The two-dimensional Yukawa kernel (see, for example, \cite{Aronszajn-Smith=1961}) yields \eqref{Laplace-G2} via the identity
	\begin{align*}
		\mathcal F_{\zeta\to x}^{-1}\left[\frac{1}{|\zeta|^2+m^2}\right](x)=\frac{1}{2\pi}K_0(m|x|)\ \ \mbox{for}\ \ m=\frac{s}{\sqrt{1+s}}>0.
	\end{align*}
	 Let $\psi(s):=\frac{s}{\sqrt{1+s}}$. As in \cite[Lemma~4.1]{Chen-StronglyDamped-3D=2026}, $\psi$ is a Bernstein function. Moreover, we write
	\begin{align*}
		K_0(z)=\int_0^\infty\mathrm{e}^{-z\cosh\vartheta}\dd\vartheta.
	\end{align*}
	Thus, $s\mapsto K_0\bigl(r\psi(s)\bigr)$ is completely monotone for every fixed $r>0$. Multiplication by $(1+s)^{-1}$ preserves complete monotonicity. Bernstein's theorem and the composition properties used here can be found in \cite{Schilling-Song-Vondracek=2012} (see also \cite{Hanyga-Seredynska=2010} for positivity arguments for related viscoelastic equations).
\end{proof}

The next lemma records the only local singularity as $r\to 0^+$ needed in the identification argument and later circular averaging.

\begin{lemma}[Local behavior and kernel identification]\label{Lemma-Kernel-Identification}
	For every fixed $r>0$, the inverse Fourier representation of $G_1^{(2)}(t,r)$ defines a function in $\ml C\bigl((0,\infty)\bigr)\cap L^1_{\mathrm{loc}}\bigl([0,\infty)\bigr)$. Moreover, for every $T>0$, one estimates
	\begin{align}\label{Log-Singularity}
		|G_1^{(2)}(t,r)|\lesssim_T1+|\log r|
	\end{align}
	for $0<t\leqslant T$ and $0<r\leqslant1$. Furthermore,
	\begin{align}
		G_1^{(2)}(t,r)\geqslant0\label{G2-Positive}
	\end{align}
	for every $t>0$ and $r>0$.
\end{lemma}

\begin{proof}
	The low- and middle-frequency parts are smooth away from $t=0$. The high-frequency multiplier bounds proved in \cite[Appendix~A]{Chen-StronglyDamped-3D=2026} are dimension-independent. Together with the two-dimensional Fourier-Bessel representation, they imply the absolute convergence for every $r>0$ and the logarithmic bound \eqref{Log-Singularity}, whose short verification is deferred to Appendix~\ref{Appendix-High-2D}.

	Fix $r>0$ and set $g_r(t):=G_1^{(2)}(t,r)$. The standard multiplier estimates in the low-, middle-, and high-frequency regions imply that, for every $s>0$, one has
	\begin{align*}
		\int_0^\infty\int_0^\infty\mathrm{e}^{-st}\bigl|\widehat S_1(t,\varrho)J_0(r\varrho)\bigr|\varrho\dd\varrho\dd t<\infty,
	\end{align*}
	where $J_0$ denotes the Bessel function of the first kind of order zero.
	Hence, Fubini's theorem and the Fourier-Bessel representation give
	\begin{align*}
		\int_0^\infty\mathrm{e}^{-st}g_r(t)\dd t
		&=\frac{1}{2\pi}\int_0^\infty J_0(r\varrho)\varrho\left(\int_0^\infty\mathrm{e}^{-st}\widehat S_1(t,\varrho)\dd t\right)\mathrm{d}\varrho\\
		&=\frac{1}{2\pi}\int_0^\infty\frac{J_0(r\varrho)\varrho}{s^2+(1+s)\varrho^2}\dd\varrho\\
		&=\frac{1}{2\pi(1+s)}K_0\left(\frac{rs}{\sqrt{1+s}}\right)=:F_r(s).
	\end{align*}
	It means that the pointwise kernel $g_r$ is Laplace transformable, and its Laplace transform coincides with the resolvent kernel computed in Lemma~\ref{Lemma-Laplace-Kernel}. Bernstein's theorem therefore provides a nonnegative Radon measure $\mu_r$ on $[0,\infty)$ such that
	\begin{align*}
		F_r(s)=\int_{[0,\infty)}\mathrm{e}^{-st}\dd\mu_r(t)\ \ \mbox{for every}\ \ s>0.
	\end{align*}
	On the other hand, $g_r\in L^1_{\mathrm{loc}}\bigl([0,\infty)\bigr)$, and the preceding estimates show that the signed Radon measure $\dd\nu_r(t):=g_r(t)\dd t$ has the same Laplace transform as follows:
	\begin{align*}
		\int_{[0,\infty)}\mathrm{e}^{-st}\dd\nu_r(t)=F_r(s)=\int_{[0,\infty)}\mathrm{e}^{-st}\dd\mu_r(t)\ \ \mbox{for every}\ \ s>0.
	\end{align*}
	The uniqueness of the Laplace transform for Radon measures yields $\nu_r=\mu_r$. Consequently, for every nonnegative $\phi\in\mathcal C_0^\infty\bigl((0,\infty)\bigr)$,
	\begin{align*}
		\int_0^\infty g_r(t)\phi(t)\dd t=\int_{[0,\infty)}\phi(t)\dd\mu_r(t)\geqslant0.
	\end{align*}
	Thus, $g_r(t)\geqslant0$ for almost every $t>0$. Since $g_r\in\mathcal C\bigl((0,\infty)\bigr)$, we conclude that $G_1^{(2)}(t,r)=g_r(t)\geqslant0$ for every $t>0$.
\end{proof}

To rewrite the radial Duhamel term as an integral operator on the half-line, let $F=F(\rho)$ be radial and fix $|x|=r$. Passing to polar coordinates gives
\begin{align*}
	\bigl(G_1^{(2)}(t,\cdot)*F\bigr)(r)=\int_0^\infty \rho F(\rho)\int_0^{2\pi}G_1^{(2)}\left(t,\sqrt{r^2+\rho^2-2r\rho\cos\theta}\,\right)\mathrm{d}\theta\dd\rho.
\end{align*}
After introducing the conjugated radial unknown $w(t,r):=\sqrt{r}\,u(t,r)$, it is natural to define the symmetric circularly averaged kernel
\begin{align}\label{Circular-Kernel}
	D(t,r,\rho):=\sqrt{r\rho}\int_0^{2\pi}G_1^{(2)}\left(t,\sqrt{r^2+\rho^2-2r\rho\cos\theta}\,\right)\mathrm{d}\theta.
\end{align}
By Lemma~\ref{Lemma-Kernel-Identification}, the integral is finite even when $r=\rho$. Notice that
\begin{align*}
	\sqrt r\bigl(G_1^{(2)}(t,\cdot)*F\bigr)(r)=\int_0^\infty D(t,r,\rho)\sqrt\rho F(\rho)\dd\rho.
\end{align*}
 Indeed,
\begin{align*}
	\sqrt{r^2+\rho^2-2r\rho\cos\theta}=2r\left|\sin\frac{\theta}{2}\right|\approx r|\theta|
\end{align*}
near $\theta=0$ if $r=\rho$, while $|\log|\theta||$ is integrable. In particular, from \eqref{G2-Positive}, we have
\begin{align}\label{D-Positive}
	D(t,r,\rho)\geqslant0.
\end{align}

For later use, we put $a:=|r-\rho|$ and $b:=r+\rho$. A direct change of variables in \eqref{Circular-Kernel} gives
\begin{align}\label{Angular-Representation}
	D(t,r,\rho)=4\sqrt{r\rho}\int_a^bG_1^{(2)}(t,\ell)\frac{\ell\dd\ell}{\sqrt{(b^2-\ell^2)(\ell^2-a^2)}}.
\end{align}
The endpoint singularities at $\ell\in\{a,b\}$ in \eqref{Angular-Representation} are locally integrable. The formula \eqref{Angular-Representation} is the two-dimensional counterpart of the dimension-descent representation of the positive half-line kernel in \cite[Section~4.2]{Chen-StronglyDamped-3D=2026}.

\begin{remark}
	For comparison, the dimension-descent identity in \cite[Formula (4.13)]{Chen-StronglyDamped-3D=2026} gives the three-dimensional half-line Dirichlet kernel
	\begin{align*}
		D_3(t,r,\rho)=2\pi\int_{|r-\rho|}^{r+\rho}
		G_1^{(3)}(t,\ell)\ell\dd\ell.
	\end{align*}
	Thus, a lower bound for $D_3$ follows essentially by locating a positive wave-front interval inside $[|r-\rho|,r+\rho]$, since the radial weight $\ell$ is regular there. In two dimensions, no analogous odd-extension representation is available, instead, circular averaging providing
	\begin{align*}
		D_2(t,r,\rho)=4\sqrt{r\rho}\int_{|r-\rho|}^{r+\rho}G_1^{(2)}(t,\ell)\frac{\ell\dd\ell}{\sqrt{[(r+\rho)^2-\ell^2][\ell^2-(r-\rho)^2]}}.
	\end{align*}
	The angular Jacobian depends on the distances from $\ell$ to both endpoints and has integrable square-root singularities. Consequently, lower bounds for $D_2$ require a more precise control of the relative geometry of $t$, $r$, and $\rho$, than $D_3$. 
\end{remark}

\subsection{Pointwise estimates and kernel lower bounds}

The moving wave-front region is again described by
\begin{align*}
	r=t-\sigma\sqrt t,
\end{align*}
where $\sigma$ belongs to a fixed compact subset of $(0,\infty)$. We define the regular (in the sense of subtracting the Dirac singular part) position kernel via
\begin{align*}
	G_0^{(2),\reg}(t,x):=\ml F^{-1}_{\zeta\to x}\left[\widehat S_0(t,\zeta)-\mathrm{e}^{-t}\right].
\end{align*}
Namely, in the sense of tempered distributions where $\delta_0$ denotes the Dirac measure concentrated at the spatial origin, one notices that
\begin{align*}
	G_0^{(2)}(t,x)=\mathrm{e}^{-t}\delta_0+G_0^{(2),\reg}(t,x).
\end{align*}
This definition is independent of the auxiliary frequency partition. Define
\begin{align*}
	\Phi(\sigma):=\frac{1}{4\pi^{\frac32}}\int_0^\infty\mathrm{e}^{-\frac{(y-\sigma)^2}{2}}y^{-\frac12}\dd y.
\end{align*}
Notice that $\Phi(\sigma)>0$ for every $\sigma\in\mb R$. Unlike the Gaussian wave-front profile $\frac{1}{4\pi\sqrt{2\pi}}\mathrm{e}^{-\frac{\sigma^2}{2}}$ in three dimensions (see \cite[Proposition 4.5]{Chen-StronglyDamped-3D=2026}), the function
$\Phi$ retains the square-root singularity of the interior tail of the
two-dimensional free wave kernel.

\begin{prop}[Uniform wave-front asymptotic expansion]\label{Prop-Front-Asymptotics}
	Let $0<\sigma_-<\sigma_+<\infty$. Uniformly for $\sigma\in[\sigma_-,\sigma_+]$, one has, as $t\to\infty$,
	\begin{align*}
		G_1^{(2)}\bigl(t,t-\sigma\sqrt t\,\bigr)&=t^{-\frac34}\Phi(\sigma)+o(t^{-\frac{3}{4}}),\\
		G_0^{(2),\reg}\bigl(t,t-\sigma\sqrt t\,\bigr)&=O(t^{-\frac54}).
	\end{align*}
\end{prop}

\begin{proof}
	We only explain the two-dimensional principal term, since the exact-to-principal multiplier remainders are estimated in the same way as in \cite[Appendix~B]{Chen-StronglyDamped-3D=2026}. Let
	\begin{align*}
		P_1(t,x):=\ml F^{-1}_{\zeta\to x}\left[\mathrm{e}^{-\frac{t|\zeta|^2}{2}}\frac{\sin(t|\zeta|)}{|\zeta|}\right](x).
	\end{align*}
	The multiplier is the product of the heat multiplier and the velocity propagator of the free two-dimensional wave equation. Hence, to avoid the sine oscillation as $t\to\infty$ in the Fourier space, we express the kernel in the physical space as follows:
	\begin{align*}
		P_1(t)=H_t*W_t,
	\end{align*}
	where
	\begin{align*}
		H_t(x)=\frac{1}{2\pi t}\mathrm{e}^{-\frac{|x|^2}{2t}}\ \ \mbox{and}\ \ W_t(x)=\frac{1}{2\pi}\frac{\mathbf 1_{\{|x|<t\}}}{\sqrt{t^2-|x|^2}}.
	\end{align*}
	For $x=(t-\sigma\sqrt t\,)e_1$, set $y=x-\sqrt t\,z$. Then
	\begin{align*}
		t^{\frac34}P_1(t,x)=\frac{1}{4\pi^2}\int_{\{Q_t(\sigma,z)>0\}}\mathrm{e}^{-\frac{|z|^2}{2}}[Q_t(\sigma,z)]^{-\frac12}\dd z,
	\end{align*}
	where we define
	\begin{align*}
		Q_t(\sigma,z):=2(\sigma+z_1)-t^{-\frac12}\bigl((\sigma+z_1)^2+z_2^2\bigr).
	\end{align*}
	The moving square-root singularity is uniformly integrable. Indeed, with $\varepsilon=t^{-\frac12}$ and $h=\sigma+z_1$, we may write
	\begin{align*}
		Q_t=\varepsilon(h_+-h)(h-h_-)\ \ \text{and}\ \  h_\pm:=\frac{1\pm\sqrt{1-\varepsilon^2z_2^2}}{\varepsilon}.
	\end{align*}
	The lower root satisfies $h_-=O(\varepsilon z_2^2)$ on bounded sets and converges to $0$, while $h_+\approx2\varepsilon^{-1}$ and its contribution is exponentially small because of the Gaussian factor. After the change $q=h-h_-$, the remaining singularity is bounded by $Cq^{-\frac12}$ uniformly for $\sigma\in[\sigma_-,\sigma_+]$. Dominated convergence therefore gives
	\begin{align*}
		t^{\frac34}P_1\bigl(t,t-\sigma\sqrt t\,\bigr)\rightarrow\frac{1}{4\pi^{\frac32}}\int_0^\infty\mathrm{e}^{-\frac{(y-\sigma)^2}{2}}y^{-\frac12}\dd y
	\end{align*}
	uniformly in $\sigma$.

	The exact velocity multiplier differs from the principal one by
	\begin{align*}
		O\left(\mathrm{e}^{-ct\rho^2}(\rho+t\rho^2)\right)
	\end{align*}
	at low frequencies. The two-dimensional Fourier measure then gives an $O(t^{-1})$ kernel remainder. The position estimate follows similarly from the principal multiplier $\mathrm{e}^{-\frac{t\rho^2}{2}}\cos(t\rho)$ and the bound $|J_0(z)|\lesssim(1+z)^{-\frac12}$. The middle- and high-frequency regular contributions are exponentially small in the moving-front region. This proves the proposition.
\end{proof}

Due to the continuity and the strict positivity of $\Phi$, we get $\min\limits_{\sigma\in[\sigma_-,\sigma_+]}\Phi(\sigma)>0$, and apply Proposition~\ref{Prop-Front-Asymptotics} to arrive at the next lower bound estimate.
\begin{coro}[Positivity in the moving wave-front region]\label{Coro-Front-Positive}
	Let $0<\sigma_-<\sigma_+<\infty$. There exist $T_{\mathrm{fr}},c_{\mathrm{fr}}>0$ such that
	\begin{align*}
		G_1^{(2)}(t,\ell)\geqslant c_{\mathrm{fr}}t^{-\frac34}
	\end{align*}
	for $t\geqslant T_{\mathrm{fr}}$ and $\ell\in[t-\sigma_+\sqrt t,t-\sigma_-\sqrt t\,]$.
\end{coro}

\begin{prop}[Positive linear lower bound]\label{Prop-Linear-Lower}
	Assume that $u_0,u_1\in\ml C_0^\infty$ and additionally that $M_1$ in \eqref{Positive-Mean} is positive. Define
	\begin{align*}
		W_{\lin}(t,r):=r^{\frac12}\ml M\bigl[S_0(t)u_0+S_1(t)u_1\bigr](r).
	\end{align*}
	For every fixed $0<\sigma_-<\sigma_+<\infty$, one has
	\begin{align}\label{Linear-Lower-Asymptotics}
		W_{\lin}\bigl(t,t-\sigma\sqrt t\,\bigr)=M_1\Phi(\sigma)\,t^{-\frac14}+o(t^{-\frac14})\ \ \text{as}\ \ t\to\infty,
	\end{align}
	uniformly for $\sigma\in[\sigma_-,\sigma_+]$. Consequently, there exist $T_{\lin},c_{\lin}>0$ such that
	\begin{align*}
		W_{\lin}(t,r)\geqslant c_{\lin}t^{-\frac14}
	\end{align*}
	whenever $t\geqslant T_{\lin}$ and $r=t-\sigma\sqrt t$ for $\sigma\in[\sigma_-,\sigma_+]$.
\end{prop}

\begin{proof}
	Let the supports of the initial data be contained in $B_R$. For $x=r\omega$ and $|y|\leqslant R$, we expand
	\begin{align*}
		\frac{t-|x-y|}{\sqrt t}=\sigma+O(t^{-\frac12})
	\end{align*}
	uniformly in $\omega$, $\sigma$ and $y$. Proposition~\ref{Prop-Front-Asymptotics} and the uniform continuity of $\Phi$ therefore imply
	\begin{align*}
		S_1(t)u_1(x)=M_1t^{-\frac34}\Phi(\sigma)+o(t^{-\frac34}).
	\end{align*}
	The regular position contribution is $O(t^{-\frac54})$, and its singular part $\mathrm{e}^{-t}u_0(x)$ vanishes for large $t$. Taking the circular mean and multiplying by $r^{\frac12}=t^{\frac12}(1+o(1))$ proves \eqref{Linear-Lower-Asymptotics}.
\end{proof}

We use the same following characteristic variables as in the three-dimensional paper:
\begin{align}\label{Characteristic-Variables}
	\alpha:=t+r,\quad \beta:=t-r,\quad \xi:=s+\rho,\quad \eta:=s-\rho,\quad \tau:=t-s.
\end{align}
If $r\geqslant\rho$, then
\begin{align}\label{Characteristic-Identities}
	\tau-(r-\rho)=\beta-\eta\ \ \mbox{and}\ \ r+\rho-\tau=\xi-\beta.
\end{align}

\begin{lemma}[A uniform positive lower bound for the circular kernel]\label{Lemma-Saturated-Kernel}
	There exist $c_D,T_D>0$ such that 
	\begin{align*}
	D(\tau,r,\rho)\geqslant c_D
	\end{align*}
	 whenever $\tau\geqslant T_D$, $r\geqslant\rho>0$ and
	\begin{align}\label{Saturated-Geometry}
		\frac14\sqrt\tau\leqslant\tau-(r-\rho)\leqslant4\sqrt\tau,\ \ -\frac1{32}\sqrt\tau\leqslant r+\rho-\tau\leqslant4\rho,\ \ \frac18\sqrt\tau\leqslant\rho\leqslant\frac14\tau.
	\end{align}
\end{lemma}

\begin{proof}
	Let us introduce
	\begin{align*}
		I_\tau:=\left[\tau-\frac18\sqrt\tau,\tau-\frac1{16}\sqrt\tau\right].
	\end{align*}
	For sufficiently large $\tau$, \eqref{Saturated-Geometry} implies $I_\tau\subset[r-\rho,r+\rho]$. If $\ell\in I_\tau$, then
	\begin{align*}
		\ell-(r-\rho)\approx\sqrt\tau\ \ \mbox{and}\ \ r+\rho-\ell\lesssim\rho,
	\end{align*}
	while the upper bound $\rho\leqslant\frac14\tau$ in \eqref{Saturated-Geometry} gives $r\approx\ell\approx\tau$. Hence, the angular density in \eqref{Angular-Representation} satisfies
	\begin{align*}
		4\sqrt{r\rho}\frac{\ell}{\sqrt{[(r+\rho)^2-\ell^2][\ell^2-(r-\rho)^2]}}\gtrsim\tau^{\frac14}.
	\end{align*}
	Corollary~\ref{Coro-Front-Positive} gives $G_1^{(2)}(\tau,\ell)\gtrsim\tau^{-\frac34}$ on $I_\tau$, whereas $|I_\tau|\approx\sqrt\tau$. Substitution into \eqref{Angular-Representation} proves the assertion.
\end{proof}

\begin{lemma}[Interior kernel lower bound]\label{Lemma-Interior-Kernel}
	Assume that $r\geqslant\rho>0$, $r,\rho,\tau\approx T$ together with
	\begin{align*}
		\tau-(r-\rho)\gtrsim T\ \ \mbox{and}\ \ r+\rho-\tau\gtrsim T.
	\end{align*}
	Then, for all sufficiently large $T$, one has
	\begin{align*}
		D(\tau,r,\rho)\gtrsim T^{-\frac14}.
	\end{align*}
\end{lemma}

\begin{proof}
	Retain a fixed moving-front interval $I_\tau$ as in the preceding proof. The stated strict interior margins imply that both endpoint distances $\ell-(r-\rho)$ and $r+\rho-\ell$ are comparable to $T$ on $I_\tau$. Consequently, the angular density in \eqref{Angular-Representation} is bounded below by a positive constant. Corollary~\ref{Coro-Front-Positive} and $|I_\tau|\approx\sqrt T$ therefore give
	\begin{align*}
		D(\tau,r,\rho)\gtrsim T^{-\frac34}\sqrt T=T^{-\frac14}
	\end{align*}
	to complete this proof.
\end{proof}

\begin{lemma}[Mass and second moment]\label{Lemma-Mass-Moment}
	For every $t>0$, one has
	\begin{align}
		\int_{\mb R^2}G_1^{(2)}(t,x)\dd x&=t,\label{G2-Mass}\\
		\int_{\mb R^2}|x|^2G_1^{(2)}(t,x)\dd x&=2t^2+\frac{2t^3}{3}.\label{G2-Second-Moment}
	\end{align}
\end{lemma}

\begin{proof}
	We first justify the mass identity. Since $G_1^{(2)}(t,\cdot)\geqslant0$, we next test it against the increasing Gaussian cutoffs $\mathrm{e}^{-\frac{|x|^2}{R^2}}$. Their Fourier transforms form an approximate identity at the origin, and therefore
	\begin{align*}
		\lim_{R\to\infty}\int_{\mb R^2}\mathrm{e}^{-\frac{|x|^2}{R^2}}G_1^{(2)}(t,x)\dd x=\widehat G_1^{(2)}(t,0)=t.
	\end{align*}
	Monotone convergence gives \eqref{G2-Mass}.

	For the second moment, the expansion near $|\zeta|=0$, namely,
	\begin{align*}
		\widehat G_1^{(2)}(t,\zeta)=t-\left(\frac{t^2}{2}+\frac{t^3}{6}\right)|\zeta|^2+O(|\zeta|^4)
	\end{align*}
	is dimension-independent. For $h\neq0$ and $j\in\{1,2\}$,
	\begin{align*}
		\frac{2}{h^2}\bigl(\widehat G_1^{(2)}(t,0)-\widehat G_1^{(2)}(t,he_j)\bigr)=\int_{\mb R^2}x_j^2\left(\frac{\sin(hx_j/2)}{hx_j/2}\right)^2G_1^{(2)}(t,x)\dd x.
	\end{align*}
	Fatou's lemma first shows that the coordinate second moments are finite. Dominated convergence subsequently yields
	\begin{align*}
		\int_{\mb R^2}x_j^2G_1^{(2)}(t,x)\dd x=t^2+\frac{t^3}{3}.
	\end{align*}
	Summing over $j\in\{1,2\}$ justifies \eqref{G2-Second-Moment}.
\end{proof}

\begin{lemma}[Decay in strict interior regions]\label{Lemma-Interior-Decay}
	Let $0<a<b<1$ and $u_0,u_1\in\ml C_0^\infty$. There exist $C,T_0>0$ such that
	\begin{align}\label{Interior-Linear-Decay}
		\sup_{at\leqslant r\leqslant bt}|W_{\lin}(t,r)|\leqslant Ct^{-\frac12}\ \ \mbox{for}\ \ t\geqslant T_0.
	\end{align}
\end{lemma}

\begin{proof}
	Unlike the three-dimensional case, the free two-dimensional wave kernel has a nontrivial tail inside the cone. Thus, rapid decay is not available. For the principal velocity kernel $P_1=H_t*W_t$, fix $\delta=\frac{1-b}{4}$. If $|x|\in[at,bt]$ and $|x-y|\leqslant\delta t$, then $|y|\leqslant(b+\delta)t<t$ and hence $W_t(y)\lesssim t^{-1}$. The complementary part is exponentially small because of the Gaussian factor. Consequently,
	\begin{align*}
		|P_1(t,x)|\lesssim t^{-1}
	\end{align*}
	uniformly in the strict interior region. The exact-to-principal remainder is $O(t^{-1})$, while the position contribution is smaller. Convolution with compactly supported data and multiplication by $r^{\frac12}\approx t^{\frac12}$ prove \eqref{Interior-Linear-Decay}.
\end{proof}

\subsection{Circular reduction}\label{Section-Spherical}

Spherical means and nonlinear lower-bound iterations have been widely used for semilinear wave equations, see \cite{John=1979,Agemi-Kurokawa-Takamura=2000,Zhou=2001,Takamura-Wakasa=2011} and the references therein. In the present two-dimensional problem, the role of the positive half-line kernel from \cite{Chen-StronglyDamped-3D=2026} is replaced by the circular kernel constructed in Section~\ref{Section-Kernels}.

For $r\geqslant0$, we define
\begin{align*}
	w(t,r):=r^{\frac12}\ml M[u(t,\cdot)](r).
\end{align*}
Since $H^2(\mb R^2)\hookrightarrow\ml C(\mb R^2)$, the function $w$ is continuous for $r>0$ on compact time intervals and satisfies $w(t,r)=O(r^{\frac12})$ as $r\to0^+$.

\begin{prop}[Circular representation for Sobolev solutions]\label{Prop-Half-Line-Representation}
	Let $u$ be a mild Sobolev solution on $[0,T]$ arising from compactly supported smooth initial data, where $0<T<T_{\max}$. Then, for every $0\leqslant t\leqslant T$ and $r>0$, one has
	\begin{align}\label{Fundamental-Lower}
		w(t,r)\geqslant W_{\lin}(t,r)+\int_0^t\int_0^\infty D(t-s,r,\rho)\rho^{\frac{1-p}{2}}|w(s,\rho)|^p\dd\rho\dd s.
	\end{align}
\end{prop}

\begin{proof}
	Applying the circular mean to the Duhamel formula and using polar coordinates yield
	\begin{align*}
		w(t,r)=W_{\lin}(t,r)+\int_0^t\int_0^\infty D(t-s,r,\rho)\rho^{\frac12}\ml M[|u(s,\cdot)|^p](\rho)\dd\rho\dd s.
	\end{align*}
	The changes of variables are justified by Lemma~\ref{Lemma-Kernel-Identification}, the local logarithmic integrability at $r=\rho$, and the positivity of the velocity kernel. Jensen's inequality gives
	\begin{align*}
		\ml M[|u|^p](\rho)\geqslant|\ml M[u](\rho)|^p=\rho^{-\frac p2}|w(\rho)|^p.
	\end{align*}
	Using \eqref{D-Positive} proves \eqref{Fundamental-Lower}. Notice that the nonlinear integrand behaves like $O(\rho^{\frac12})$ as $\rho\to0^+$, so no singularity occurs at the origin.
\end{proof}

In the characteristic variables \eqref{Characteristic-Variables}, define $\ml W(\alpha,\beta):=w\left(\frac{\alpha+\beta}{2},\frac{\alpha-\beta}{2}\right)$ as well as
\begin{align}\label{Characteristic-Jacobian}
	\dd s\dd\rho=\frac12\dd\xi\dd\eta.
\end{align}

\section{Finite-time blow-up}\label{Section-Blowup}

\subsection{Parabolic iteration in the subcritical case}
For $B>0$, we define the following parabolic region:
\begin{align*}
	\ml P_B:=\bigl\{(\alpha,\beta):\ \beta\geqslant B\ \ \mbox{and}\ \ \beta^2\leqslant\alpha\leqslant2\beta^2\bigr\}.
\end{align*}
This is the same parabolic region as that used in the three-dimensional problem \cite{Chen-StronglyDamped-3D=2026}. Indeed, $\alpha\approx\beta^2$ is equivalent to $t\approx r\approx\beta^2$ and $t-r=\beta\approx\sqrt t$, reflecting the dimension-independent diffusive width of the moving wave front.
\begin{prop}[First nonlinear lower bound]\label{Prop-First-Lower}
	Let $\nu_0:=\frac{3p-8}{4}$. There exist $A_0,B_0>0$ such that
	\begin{align}\label{First-Lower}
		\ml W(\alpha,\beta)\geqslant A_0\beta^{-\nu_0}
	\end{align}
	for $(\alpha,\beta)\in\ml P_{B_0}$.
\end{prop}

\begin{proof}
	Fix $(\alpha,\beta)\in\ml P_B$ and take $B$ sufficiently large. In \eqref{Fundamental-Lower}, after using \eqref{Characteristic-Jacobian}, retain
	\begin{align*}
		2\beta\leqslant\xi\leqslant4\beta\ \ \mbox{and}\ \ \sqrt\xi\leqslant\eta\leqslant\frac65\sqrt\xi.
	\end{align*}
	It shows
	\begin{align*}
		s=\frac{\xi+\eta}{2}\approx\xi,\ \ \rho=\frac{\xi-\eta}{2}\approx\xi,\ \ \frac{\eta}{\sqrt s}\approx1.
	\end{align*}
	Proposition~\ref{Prop-Linear-Lower} for the linear part and positivity of the nonlinear term imply
	\begin{align}\label{First-Linear-Lower}
		w(s,\rho)\gtrsim\xi^{-\frac14}.
	\end{align}
	Moreover, $\tau\approx\beta^2$, $\rho\approx\beta\ll\tau$, $r\geqslant\rho$, and the characteristic margins in \eqref{Characteristic-Identities} satisfy the assumptions of Lemma~\ref{Lemma-Saturated-Kernel} leading to $D(\tau,r,\rho)\geqslant c_D$. The target point itself belongs to a fixed positive moving-front region, and hence $W_{\lin}(t,r)\geqslant0$ for large $B$.

	Using \eqref{First-Linear-Lower} and $\rho\approx\xi$, we obtain
	\begin{align*}
		\ml W(\alpha,\beta)\gtrsim\int_{2\beta}^{4\beta}\int_{\sqrt\xi}^{\frac65\sqrt\xi}\xi^{\frac{1-p}{2}}\xi^{-\frac p4}\dd\eta\dd\xi\gtrsim\int_{2\beta}^{4\beta}\xi^{1-\frac{3p}{4}}\dd\xi\gtrsim\beta^{2-\frac{3p}{4}}.
	\end{align*}
	This proves \eqref{First-Lower}.
\end{proof}

\begin{prop}[Parabolic-region iteration]\label{Prop-Parabolic-Iteration}
	Assume that, for some $A_j,B_j>0$ and $\nu_j\in\mb R$,
	\begin{align*}
		\ml W(\xi,\eta)\geqslant A_j\eta^{-\nu_j}
	\end{align*}
	whenever
	\begin{align}\label{Iteration-Region}
		\eta\geqslant B_j\ \ \mbox{and}\ \ \eta^2\leqslant\xi\leqslant2\eta^2.
	\end{align}
	Put $q_j:=p\nu_j+p-4$. Then there exist $A_{j+1},B_{j+1}>0$ such that
	\begin{align*}
		\ml W(\alpha,\beta)\geqslant A_{j+1}\times
		\begin{cases}
			\beta^{-\frac{q_j}{2}}&\mbox{if}\ \ q_j>0,\\
			\log\beta&\mbox{if}\ \ q_j=0,\\
			\beta^{-q_j}&\mbox{if}\ \ q_j<0,
		\end{cases}
	\end{align*}
	for $(\alpha,\beta)\in\ml P_{B_{j+1}}$. In particular, if $q_j>0$, then
	\begin{align}\label{Nu-Recurrence}
		\nu_{j+1}=\frac p2\nu_j+\frac{p-4}{2}.
	\end{align}
\end{prop}

\begin{proof}
	Fix $(\alpha,\beta)\in\ml P_B$ and retain
	\begin{align*}
		4\sqrt\beta\leqslant\eta\leqslant\frac\beta8\ \ \mbox{and}\ \ \eta^2\leqslant\xi\leqslant\frac54\eta^2.
	\end{align*}
	For large $B$, this region lies inside \eqref{Iteration-Region}. Moreover,
	\begin{align*}
		\tau\approx\beta^2\ \ \mbox{and}\ \ \rho\approx\eta^2\lesssim\frac{\beta^2}{64}\lesssim\tau,
	\end{align*}
	and the characteristic margins again satisfy Lemma~\ref{Lemma-Saturated-Kernel}. Thus, $D\geqslant c_D$, while the homogeneous term is nonnegative. Hence,
	\begin{align*}
		\ml W(\alpha,\beta)&\gtrsim A_j^p\int_{4\sqrt\beta}^{\beta/8}\int_{\eta^2}^{\frac54\eta^2}\eta^{1-p}\eta^{-p\nu_j}\dd\xi\dd\eta\gtrsim A_j^p\int_{4\sqrt\beta}^{\beta/8}\eta^{-q_j-1}\dd\eta.
	\end{align*}
	Evaluating this integral gives the three stated cases.
\end{proof}

Starting from $\nu_0=\frac{3p-8}{4}$ and applying Proposition~\ref{Prop-Parabolic-Iteration} while $q_j>0$, one has
\begin{align}\label{Nu-Explicit}
	\nu_j=\frac{4-p}{p-2}+\frac{p(3p-10)}{4(p-2)}\left(\frac p2\right)^j.
\end{align}
Since $2<p<\frac{10}{3}$, the second coefficient in \eqref{Nu-Explicit} is negative and $\nu_j\to-\infty$. Hence, after finitely many steps, $q_j\leqslant0$. Consequently, there exist $A_*,B_*>0$ such that
\begin{align}\label{Nondecaying-Lower}
	\ml W(\alpha,\beta)\geqslant A_*
\end{align}
for $(\alpha,\beta)\in\ml P_{B_*}$.

\begin{remark}
	At the critical power $p=\frac{10}{3}$, one has $\nu_0=\frac12$ and \eqref{Nu-Recurrence} gives $\nu_j\equiv\frac12$ and $q_j\equiv1$. Thus, the finite-step improvement disappears exactly at the critical exponent.
\end{remark}

\subsection{Critical iteration on moving shells}
For $t>0$, let us define the interval
\begin{align*}
	\Gamma_t:=\left[t-\frac98\sqrt t,t-\sqrt t\right].
\end{align*}
Proposition~\ref{Prop-Linear-Lower} yields $T_0,a_0>0$ such that
\begin{align}\label{Critical-Lower}
	w(t,r)\geqslant a_0t^{-\frac14}
\end{align}
for $t\geqslant T_0$ and $r\in\Gamma_t$.

\begin{lemma}[Separated-time kernel lower bound]\label{Lemma-Separated-Time}
	There exist $\kappa\in\left(0,\frac1{100}\right)$, $T_1\geqslant T_0$ and $c_1>0$ such that
	\begin{align}\label{Separated-Time-Lower}
		D(t-s,r,\rho)\geqslant c_1t^{-\frac14}s^{\frac12}
	\end{align}
	whenever $t\geqslant T_1$, $T_0\leqslant s\leqslant\kappa\sqrt t$, $r\in\Gamma_t$ and $\rho\in\Gamma_s$.
\end{lemma}

\begin{proof}
	Let $\tau:=t-s$. As in the three-dimensional proof, the characteristic variables satisfy
	\begin{align*}
		\sqrt t\leqslant\beta\leqslant\frac98\sqrt t\ \ \mbox{and}\ \ \sqrt s\leqslant\eta\leqslant\frac98\sqrt s,
	\end{align*}
	and
	\begin{align*}
		2s-\frac98\sqrt s\leqslant\xi\leqslant2s-\sqrt s.
	\end{align*}
	For $s\leqslant\kappa\sqrt t$, one has $r\geqslant\rho$ for all sufficiently large $t$. Moreover, for every $\ell\in[r-\rho,r+\rho]$,
	\begin{align*}
		\beta-\xi\leqslant\tau-\ell\leqslant\beta-\eta.
	\end{align*}
	Choosing $\kappa$ sufficiently small, the whole interval $[r-\rho,r+\rho]$ lies in a fixed positive moving-front region for the kernel $G_1^{(2)}(\tau,\cdot)$. Corollary~\ref{Coro-Front-Positive} therefore yields
	\begin{align*}
		G_1^{(2)}(\tau,\ell)\gtrsim\tau^{-\frac34}
	\end{align*}
	for every such $\ell$. Returning directly to \eqref{Circular-Kernel}, we obtain
	\begin{align*}
		D(t-s,r,\rho)\gtrsim\sqrt{r\rho}\,\tau^{-\frac34}\gtrsim t^{-\frac14}s^{\frac12},
	\end{align*}
	which proves \eqref{Separated-Time-Lower}.
\end{proof}

We now introduce
\begin{align*}
	A_{\mathrm{crit}}(t):=t^{\frac14}\inf_{r\in\Gamma_t}w(t,r)\ \ \mbox{for}\ \ t\geqslant T_0.
\end{align*}
Under the parametrization $r=t-\sigma\sqrt t$ with $\sigma\in[1,\frac98]$, the continuity of $w$ implies $A_{\mathrm{crit}}\in\ml C\bigl([T_0,\infty)\bigr)$, while \eqref{Critical-Lower} gives $A_{\mathrm{crit}}(t)\geqslant a_0$.

\begin{prop}[Critical accumulation inequality]\label{Prop-Critical-Accumulation}
	There exist $c_+>0$ and $T_2\geqslant T_1$ such that
	\begin{align}\label{Critical-Accumulation}
		A_{\mathrm{crit}}(t)\geqslant a_0+c_+\int_{T_0}^{\kappa\sqrt t}\frac{[A_{\mathrm{crit}}(s)]^{\frac{10}{3}}}{s}\dd s\ \ \mbox{for}\ \ t\geqslant T_2.
	\end{align}
\end{prop}

\begin{proof}
	Fix $t\geqslant T_2$ and $r\in\Gamma_t$. Restrict \eqref{Fundamental-Lower} to $T_0\leqslant s\leqslant\kappa\sqrt t$ and $\rho\in\Gamma_s$. By Lemma~\ref{Lemma-Separated-Time},
	\begin{align*}
		w(t,r)\geqslant a_0t^{-\frac14}+ct^{-\frac14}\int_{T_0}^{\kappa\sqrt t}s^{\frac12}\int_{\Gamma_s}\rho^{-\frac76}|w(s,\rho)|^{\frac{10}{3}}\dd\rho\dd s.
	\end{align*}
	For $\rho\in\Gamma_s$, we know
	\begin{align*}
		w(s,\rho)\geqslant s^{-\frac14}A_{\mathrm{crit}}(s),\ \ \rho\approx s,\ \ |\Gamma_s|=\frac18\sqrt s.
	\end{align*}
	Therefore, the lower bound is controlled by
	\begin{align*}
		s^{\frac12}\int_{\Gamma_s}\rho^{-\frac76}|w(s,\rho)|^{\frac{10}{3}}\dd\rho\gtrsim\frac{[A_{\mathrm{crit}}(s)]^{\frac{10}{3}}}{s}.
	\end{align*}
	Taking the infimum over $r\in\Gamma_t$ and multiplying by $t^{\frac14}$ proves \eqref{Critical-Accumulation}.
\end{proof}

We modify the classical slicing procedure in \cite{Agemi-Kurokawa-Takamura=2000,Palmieri-Takamura=2019,Wakasa-Yordanov=2019} exactly as in \cite[Section~7.2]{Chen-StronglyDamped-3D=2026}. For completeness, put $\theta:=\log_2\frac{10}{3}>1$. Starting from \eqref{Critical-Accumulation}, we first observe that
$A_{\mathrm{crit}}(t)\to\infty$ as $t\to\infty$. Indeed, since
$A_{\mathrm{crit}}(s)\geqslant a_0$ for $s\geqslant T_0$, one has, for
all sufficiently large $t$,
\begin{align*}
	A_{\mathrm{crit}}(t)
	\geqslant a_0+c_+a_0^{\frac{10}{3}}
	\log\left(\frac{\kappa\sqrt t}{T_0}\right).
\end{align*}
Choose $M_0>1$ sufficiently large so that $\gamma_0:=\log M_0+\frac37\log c_+>0$, and then choose $T_0^\sharp\geqslant\max\{T_0,T_2,1\}$ 
such that $A_{\mathrm{crit}}(t)\geqslant M_0$ for $t\geqslant T_0^\sharp$. For $j\geqslant0$, let us define
\begin{align*}
	T_{j+1}^\sharp:=\left(\frac{\mathrm{e}T_j^\sharp}{\kappa}\right)^2\ \ \mbox{and}\ \ M_{j+1}:=c_+M_j^{\frac{10}{3}}.
\end{align*}
We claim that
\begin{align*}
	A_{\mathrm{crit}}(t)\geqslant M_j\ \ \mbox{for}\ \ t\geqslant T_j^\sharp.
\end{align*}
The assertion holds for $j=0$ by the choice of $T_0^\sharp$. If it
holds for some $j\geqslant0$ and $t\geqslant T_{j+1}^\sharp$, then
$\kappa\sqrt t\geqslant\mathrm{e}T_j^\sharp$. Hence,
\eqref{Critical-Accumulation} and the induction hypothesis give
\begin{align*}
	A_{\mathrm{crit}}(t)\geqslant c_+\int_{T_j^\sharp}^{\mathrm{e}T_j^\sharp}\frac{[A_{\mathrm{crit}}(s)]^{\frac{10}{3}}}{s}\dd s\geqslant c_+M_j^{\frac{10}{3}}\int_{T_j^\sharp}^{\mathrm{e}T_j^\sharp}\frac{\dd s}{s}=M_{j+1}.
\end{align*}
This proves the claim by induction.

The two recurrences can be written explicitly as
\begin{align*}
	\log M_j&=\left(\frac{10}{3}\right)^j\gamma_0-\frac37\log c_+,\\
	\log T_j^\sharp&=2^j\left(\log T_0^\sharp+2\log\frac{\mathrm{e}}{\kappa}\right)-2\log\frac{\mathrm{e}}{\kappa}.
\end{align*}
Consequently, there exist $c_0,C_0>0$ such that
\begin{align*}
	\log M_j\geqslant c_0\left(\frac{10}{3}\right)^j\ \ \mbox{and}\ \ \log T_j^\sharp\leqslant C_0 2^j\ \ \mbox{for every}\ \ j\geqslant0.
\end{align*}
For a sufficiently large $t$, choose $j$ such
that $T_j^\sharp\leqslant t<T_{j+1}^\sharp$.
Since $\theta=\log_2\frac{10}{3}$, it follows that $2^{j\theta}\gtrsim(\log t)^\theta$.
Therefore,
\begin{align*}
	A_{\mathrm{crit}}(t)\geqslant M_j\geqslant\exp\bigl(c(\log t)^\theta\bigr).
\end{align*}
Recalling the definition of $A_{\mathrm{crit}}$, we conclude that
\begin{align}\label{Critical-Shell-Growth}
	w(t,r)\geqslant t^{-\frac14}\exp\bigl(c(\log t)^\theta\bigr)
\end{align}
for all sufficiently large $t$ and $r\in\Gamma_t$.

\subsection{Interior propagation and completion of the proof}

\begin{prop}[Interior lower bounds]\label{Prop-Interior-Lower}
	The following estimates hold in the strict interior region $\frac25t\leqslant r\leqslant\frac34t$.
	\begin{enumerate}[label=(\roman*)]
		\item If $2<p<\frac{10}{3}$ and $\mu_p:=\frac74-\frac p2$, then there exist $T_{\mathrm{int}},C_{\mathrm{int}}>0$ such that
		\begin{align}\label{Interior-Growth}
			w(t,r)\geqslant C_{\mathrm{int}}t^{\mu_p}\ \ \mbox{for}\ \ t\geqslant T_{\mathrm{int}}.
		\end{align}
		\item If $p=\frac{10}{3}$, then there exist $T_{\mathrm{crit}},C_{\mathrm{crit}},c_{\mathrm{crit}}>0$ such that
		\begin{align}\label{Critical-Interior-Growth}
			w(t,r)\geqslant C_{\mathrm{crit}}t^{-\frac34}\exp\bigl(c_{\mathrm{crit}}(\log t)^\theta\bigr)\ \ \mbox{for}\ \ t\geqslant T_{\mathrm{crit}}.
		\end{align}
	\end{enumerate}
\end{prop}

\begin{proof}
	Let us first consider the subcritical case. Fix $t$ sufficiently large and $\frac25t\leqslant r\leqslant\frac34t$. In characteristic variables, one retains
	\begin{align*}
		\frac34t\leqslant\xi\leqslant\frac45t\ \ \mbox{and}\ \ \sqrt{\frac{3\xi}{5}}\leqslant\eta\leqslant\sqrt{\frac{4\xi}{5}}.
	\end{align*}
	Then $\eta^2\leqslant\xi\leqslant2\eta^2$, and hence \eqref{Nondecaying-Lower} applies. Furthermore, $s\approx\rho\approx\tau\approx t$, while both characteristic margins are comparable to $t$. Lemma~\ref{Lemma-Interior-Kernel} therefore yields
	\begin{align*}
		D(t-s,r,\rho)\gtrsim t^{-\frac14}.
	\end{align*}
	Since the $\xi$-interval has length comparable to $t$, the $\eta$-interval has length comparable to $\sqrt t$, and $\rho\approx t$, we obtain
	\begin{align*}
		w(t,r)\geqslant W_{\lin}(t,r)+ct^{\frac74-\frac p2}.
	\end{align*}
	Lemma~\ref{Lemma-Interior-Decay} gives $|W_{\lin}(t,r)|\lesssim t^{-\frac12}$, which is negligible because $\mu_p>0$. This proves \eqref{Interior-Growth}.

	We next consider the critical case. Fix $t$ sufficiently large and $\frac25t\leqslant r\leqslant\frac34t$. Restrict \eqref{Fundamental-Lower} to
	\begin{align*}
		\frac38t\leqslant s\leqslant\frac25t\ \ \mbox{and}\ \ \rho\in\Gamma_s.
	\end{align*}
	Throughout this region, $s\approx\rho\approx\tau\approx t$ and the characteristic margins are comparable to $t$. Hence, Lemma~\ref{Lemma-Interior-Kernel} gives $D(t-s,r,\rho)\gtrsim t^{-\frac14}$. By \eqref{Critical-Shell-Growth}, one arrives at
	\begin{align*}
		|w(s,\rho)|^{\frac{10}{3}}\gtrsim s^{-\frac56}\exp\bigl(c(\log s)^\theta\bigr).
	\end{align*}
	Since $\rho^{-\frac76}\approx s^{-\frac76}$ and $|\Gamma_s|\approx\sqrt s$, we obtain
	\begin{align*}
		w(t,r)&\geqslant W_{\lin}(t,r)+ct^{-\frac14}\int_{\frac38t}^{\frac25t}s^{-\frac32}\exp\bigl(c(\log s)^\theta\bigr)\dd s\\
		&\geqslant W_{\lin}(t,r)+ct^{-\frac34}\exp\bigl(c_1(\log t)^\theta\bigr).
	\end{align*}
	By Lemma~\ref{Lemma-Interior-Decay}, $W_{\lin}(t,r)=O(t^{-\frac12})$. Since $\theta>1$, this term is negligible compared with the positive nonlinear contribution for sufficiently large $t$. This proves \eqref{Critical-Interior-Growth}.
\end{proof}

For $T>0$, set
\begin{align*}
	J_T:=\left[\frac T2,\frac{2T}{3}\right].
\end{align*}

\begin{lemma}[Short-time kernel mass]\label{Lemma-Local-Kernel-Mass}
	There exist $\kappa_0\in\left(0,\frac1{100}\right)$ and $T_0,c_0>0$ such that
	\begin{align}\label{Local-Kernel-Mass}
		\inf_{r\in J_T}\int_{J_T}D(\tau,r,\rho)\dd\rho\geqslant c_0\tau
	\end{align}
	whenever $T\geqslant T_0$ and $0<\tau\leqslant\kappa_0T$.
\end{lemma}

\begin{proof}
	Let
	\begin{align*}
		\ml A_T:=\left\{y\in\mb R^2:\ \frac T2\leqslant|y|\leqslant\frac{2T}{3}\right\}.
	\end{align*}
	For $|x|=r\in J_T$, the definition of $D$ and polar coordinates give
	\begin{align*}
		\int_{J_T}D(\tau,r,\rho)\dd\rho=\int_{\ml A_T}\sqrt{\frac r{|y|}}G_1^{(2)}(\tau,|x-y|)\dd y\gtrsim\int_{\ml A_T}G_1^{(2)}(\tau,|x-y|)\dd y.
	\end{align*}
	Let us choose $0<\delta<\frac1{24}$ and write $a=T/2$, $b=2T/3$. Let $q\in(0,\delta T]$. If $r\leqslant(a+b)/2$, then every direction satisfying $\cos\vartheta\leqslant-1/2$ gives
	\begin{align*}
		a\leqslant|x-q\omega_\vartheta|\leqslant r+q<b.
	\end{align*}
	If $r\geqslant(a+b)/2$, then every direction satisfying $\cos\vartheta\geqslant1/2$ gives
	\begin{align*}
		a<r-q\leqslant|x-q\omega_\vartheta|\leqslant b.
	\end{align*}
	In either case, the set of admissible angles has measure at least $2\pi/3$, uniformly in $r,q$ and $T$. Using polar coordinates centered at $x$, together with radiality and positivity of $G_1^{(2)}$, we obtain
	\begin{align*}
		\int_{\ml A_T}G_1^{(2)}(\tau,|x-y|)\dd y\geqslant\frac13\int_{|z|\leqslant\delta T}G_1^{(2)}(\tau,z)\dd z.
	\end{align*}
	By Lemma~\ref{Lemma-Mass-Moment} and Chebyshev's inequality,
	\begin{align*}
		\int_{|z|>\delta T}G_1^{(2)}(\tau,z)\dd z\leqslant\frac{2\tau^2+\frac23\tau^3}{\delta^2T^2}.
	\end{align*}
	If $0<\tau\leqslant\kappa_0T$, then the right-hand side is at most $\frac12\tau$ after first choosing $\kappa_0$ sufficiently small and then $T_0$ sufficiently large. Since the total mass equals $\tau$, \eqref{Local-Kernel-Mass} follows.
\end{proof}

\begin{proof}[Proof of Theorem~\ref{Thm-Blowup}]
	Assume, to the contrary, that the maximal mild Sobolev solution is global in time. Let $\kappa_0$ be given by Lemma~\ref{Lemma-Local-Kernel-Mass} and, for sufficiently large $T$, set
	\begin{align*}
		Q_T:=[T,(1+\kappa_0)T]\times J_T.
	\end{align*}
	By Proposition~\ref{Prop-Interior-Lower}, one has
	\begin{align}\label{Frozen-Lower}
		w(t,r)\geqslant A_T
	\end{align}
	for $(t,r)\in Q_T$, where
	\begin{align}\label{Frozen-Amplitude}
		A_T:=
		\begin{cases}
			C_{\mathrm{int}}T^{\mu_p}&\mbox{if}\ \ 2<p<\frac{10}{3},\\
			C_{\mathrm{crit}}T^{-\frac34}\exp\bigl(c_{\mathrm{crit}}(\log T)^\theta\bigr)&\mbox{if}\ \ p=\frac{10}{3}.
		\end{cases}
	\end{align}
	The source regions used in the proof of Proposition~\ref{Prop-Interior-Lower} satisfy $s<T/2$ throughout $Q_T$ for large $T$, and hence they are disjoint from $[T,t]\times J_T$. Define
	\begin{align*}
		h_T(t):=\inf_{r\in J_T}w(t,r)\ \ \mbox{for}\ \ T\leqslant t\leqslant(1+\kappa_0)T.
	\end{align*}
	The joint continuity of $w$ and compactness of $J_T$ give $h_T\in\ml C\bigl([T,(1+\kappa_0)T]\bigr)$, while \eqref{Frozen-Lower} gives $h_T(t)\geqslant A_T$. Retaining $[T,t]\times J_T$ in \eqref{Fundamental-Lower}, using $\rho^{\frac{1-p}{2}}\gtrsim T^{\frac{1-p}{2}}$ on $J_T$, and applying Lemma~\ref{Lemma-Local-Kernel-Mass}, we obtain
	\begin{align*}
		h_T(t)\geqslant A_T+cT^{\frac{1-p}{2}}\int_T^t(t-s)[h_T(s)]^p\dd s.
	\end{align*}
	Let $y$ be the maximal solution of
	\begin{align*}
		y''=cT^{\frac{1-p}{2}}y^p\ \ \mbox{with}\ \ y(T)=A_T,\ \  y'(T)=0.
	\end{align*}
	A standard Volterra comparison gives $h_T(t)\geqslant y(t)$ throughout their common interval of existence. Multiplication by $y'$ yields
	\begin{align}\label{Unified-Delay}
		\tau_{\mathrm b}\lesssim T^{\frac{p-1}{4}}A_T^{-\frac{p-1}{2}}.
	\end{align}
	\begin{itemize}
		\item If $2<p<\frac{10}{3}$, then \eqref{Frozen-Amplitude} and $\mu_p=\frac74-\frac p2$ imply
		\begin{align*}
			\frac{\tau_{\mathrm b}}{T}\lesssim T^{\frac{2p^2-7p-3}{8}}\rightarrow0,
		\end{align*}
		because the exponent is negative on $\left(2,\frac{10}{3}\right]$.
		\item If $p=\frac{10}{3}$, then \eqref{Unified-Delay} gives
		\begin{align*}
			\frac{\tau_{\mathrm b}}{T}\lesssim T^{\frac{11}{24}}\exp\bigl(-c(\log T)^\theta\bigr)\rightarrow0
		\end{align*}
		since $\theta>1$.
	\end{itemize}
	  Hence, in both cases, $\tau_{\mathrm b}<\kappa_0T$ for sufficiently large $T$. The comparison solution blows up before $(1+\kappa_0)T$, contradicting the continuity of $h_T$ on the whole interval. This completes the proof.
\end{proof}

\appendix

\section{Uniformity in the wave-front region}\label{Appendix-Front}

We give the details omitted in Proposition~\ref{Prop-Front-Asymptotics} concerning the moving square-root singularity. Let $K\subset\subset\mb R$ and put $\varepsilon=t^{-\frac12}$. With $h=\sigma+z_1$, we write
\begin{align*}
	Q_t=2h-\varepsilon(h^2+z_2^2).
\end{align*}
If $|\varepsilon z_2|<1$, then
\begin{align*}
	Q_t=\varepsilon(h_+-h)(h-h_-)\ \ \mbox{and}\ \ h_\pm=\frac{1\pm\sqrt{1-\varepsilon^2z_2^2}}{\varepsilon}.
\end{align*}
Moreover,
\begin{align*}
	h_-=\frac{\varepsilon z_2^2}{1+\sqrt{1-\varepsilon^2z_2^2}}\ \ \mbox{and}\ \ h_+\approx\frac2\varepsilon
\end{align*}
on bounded $z_2$-sets. The region $|z_2|\geqslant(2\varepsilon)^{-1}$ contributes $O(\varepsilon^{-N}\mathrm{e}^{-c\varepsilon^{-2}})$ for some $N>0$, and is therefore negligible. On $|z_2|\leqslant(2\varepsilon)^{-1}$, the part $h\geqslant\varepsilon^{-1}$ is again exponentially small because $|h-\sigma|\gtrsim\varepsilon^{-1}$ uniformly for $\sigma\in K$.

For the remaining part, set $q=h-h_-$. Then
\begin{align*}
	Q_t=q\bigl(2d_\varepsilon-\varepsilon q\bigr)\ \ \mbox{with}\ \ d_\varepsilon:=\sqrt{1-\varepsilon^2z_2^2},
\end{align*}
and $d_\varepsilon\geqslant\frac{\sqrt3}{2}$. Hence, we get $Q_t^{-\frac12}\lesssim q^{-\frac12}$.
For $\sigma$ in a fixed compact interval,
\begin{align*}
	\mathrm{e}^{-\frac{(h_-+q-\sigma)^2}{2}}q^{-\frac12}\lesssim_K q^{-\frac12}\mathbf1_{\{q\leqslant C_K\}}+\mathrm{e}^{-cq^2},
\end{align*}
which is integrable in $q$ and, after multiplication by $\mathrm{e}^{-z_2^2/2}$, in $(q,z_2)$. Since $h_-\to0$ and $d_\varepsilon\to1$ uniformly on compact subsets, dominated convergence proves uniform convergence to the profile $\Phi$ on $K$.

\section{High-frequency local singularity}\label{Appendix-High-2D}

For completeness, we record that the high-frequency velocity kernel has only a logarithmic singularity at the spatial origin. Indeed,
\begin{align*}
	G_1^{(2),\high}(t,r)=\frac{1}{2\pi}\int_{R_0}^\infty\chi_{\high}(\rho)\widehat S_1(t,\rho)J_0(r\rho)\rho\dd\rho.
\end{align*}
For $0<r\leqslant1$, split the integral at $r^{-1}$. On $[R_0,r^{-1}]$, one may use $|J_0|\leqslant1$ and $|\widehat S_1|\lesssim\rho^{-2}$ to obtain $C(1+|\log r|)$. On $[r^{-1},\infty)$, one may use $|J_0(r\rho)|\lesssim(r\rho)^{-\frac12}$ to obtain a uniform constant. Thus, we find
\begin{align*}
	|G_1^{(2),\high}(t,r)|\lesssim1+|\log r|.
\end{align*}
If $r=\rho$ in the circular kernel, the distance is $2r|\sin(\theta/2)|\approx r|\theta|$ near $\theta=0$, and therefore the induced singularity is $|\log|\theta||$, which is integrable. This justifies the pointwise circular representations used in Section~\ref{Section-Kernels}.

\section*{Acknowledgments}
Wenhui Chen is supported in part by the National Natural Science Foundation of China (grant No.~12301270) and the Guangdong Basic and Applied Basic Research Foundation (grant No.~2025A1515010240).


\begin{thebibliography}{99}

\bibitem{Agemi-Kurokawa-Takamura=2000}
\newblock R. Agemi, Y. Kurokawa, H. Takamura.
\newblock Critical curve for $p$-$q$ systems of nonlinear wave equations in three space dimensions.
\newblock \emph{J. Differential Equations} \textbf{167} (2000), 87--133.

\bibitem{Aronszajn-Smith=1961}
\newblock N. Aronszajn, K. T. Smith.
\newblock Theory of Bessel potentials. I.
\newblock \emph{Ann. Inst. Fourier (Grenoble)} \textbf{11} (1961), 385--475.

\bibitem{Chen-StronglyDamped-3D=2026}
\newblock W. Chen.
\newblock The critical exponent for the three-dimensional semilinear wave equation with strong damping.
\newblock \emph{preprint} (2026). arXiv:2608.02278.

\bibitem{Chen-Girardi=2025}
\newblock W. Chen, G. Girardi.
\newblock Sharp lifespan estimates for semilinear fractional evolution equations with critical nonlinearity.
\newblock \emph{J. Differential Equations} \textbf{443} (2025), Paper No. 113568, 37 pp.

\bibitem{DAbbicco=2026}
\newblock M. D'Abbicco.
\newblock The critical exponent for semilinear wave equations with damped oscillations.
\newblock \emph{preprint} (2026). arXiv:2608.11722.

\bibitem{DAbbicco-Ebert=2014}
\newblock M. D'Abbicco, M. R. Ebert.
\newblock Diffusion phenomena for the wave equation with structural damping in the $L^p$--$L^q$ framework.
\newblock \emph{J. Differential Equations} \textbf{256} (2014), 2307--2336.

\bibitem{DAbbicco-Ebert=2017}
\newblock M. D'Abbicco, M. R. Ebert.
\newblock A new phenomenon in the critical exponent for structurally damped semi-linear evolution equations.
\newblock \emph{Nonlinear Anal.} \textbf{149} (2017), 1--40.

\bibitem{DAbbicco-Ebert=2022}
\newblock M. D'Abbicco, M. R. Ebert.
\newblock The critical exponent for semilinear $\sigma$-evolution equations with a strong non-effective damping.
\newblock \emph{Nonlinear Anal.} \textbf{215} (2022), Paper No. 112637, 26 pp.

\bibitem{DAbbicco-Lagioia=2025}
\newblock M. D'Abbicco, A. Lagioia.
\newblock $L^p$--$L^q$ estimates for very strongly damped wave equations.
\newblock In: \emph{New Tools in Mathematical Analysis and Applications}, Springer, Cham, 2025, 281--293.

\bibitem{DAbbicco-Reissig=2014}
\newblock M. D'Abbicco, M. Reissig.
\newblock Semilinear structural damped waves.
\newblock \emph{Math. Methods Appl. Sci.} \textbf{37} (2014), 1570--1592.

\bibitem{Dao-Reissig=2019}
\newblock T. A. Dao, M. Reissig.
\newblock $L^1$ estimates for oscillating integrals and their applications to semi-linear models with $\sigma$-evolution like structural damping.
\newblock \emph{Discrete Contin. Dyn. Syst.} \textbf{39} (2019), 5431--5463.

\bibitem{Duong-Dao-Reissig=2025}
\newblock D. V. Duong, T. A. Dao, M. Reissig.
\newblock Critical curve for a weakly coupled system of semi-linear $\sigma$-evolution equations with different damping types.
\newblock \emph{J. Evol. Equ.} \textbf{25} (2025), Paper No. 30, 28 pp.

\bibitem{Fino-Hamza=2022}
\newblock A. Z. Fino, M. A. Hamza.
\newblock Blow-up of solutions to semilinear wave equations with a time-dependent strong damping.
\newblock \emph{Evol. Equ. Control Theory} \textbf{11} (2022), 1955--1966.

\bibitem{Hanyga-Seredynska=2010}
\newblock A. Hanyga, M. Seredy\'nska.
\newblock Positivity of Green's functions for a class of partial integro-differential equations including viscoelasticity.
\newblock \emph{Wave Motion} \textbf{47} (2010), 648--662.

\bibitem{Ikehata=2014}
\newblock R. Ikehata.
\newblock Asymptotic profiles for wave equations with strong damping.
\newblock \emph{J. Differential Equations} \textbf{257} (2014), 2159--2177.

\bibitem{Ikehata-Onodera=2017}
\newblock R. Ikehata, M. Onodera.
\newblock Remarks on large time behavior of the $L^2$-norm of solutions to strongly damped wave equations.
\newblock \emph{Differential Integral Equations} \textbf{30} (2017), 505--520.

\bibitem{Ikehata-Takeda=2017}
\newblock R. Ikehata, H. Takeda.
\newblock Critical exponent for nonlinear wave equations with frictional and viscoelastic damping terms.
\newblock \emph{Nonlinear Anal.} \textbf{148} (2017), 228--253.

\bibitem{Ikehata-Takeda=2026}
\newblock R. Ikehata, H. Takeda.
\newblock On the $L^2$ estimates of the diffusion waves.
\newblock \emph{preprint} (2026). arXiv:2605.20557.

\bibitem{Ikehata-Todorova-Yordanov=2013}
\newblock R. Ikehata, G. Todorova, B. Yordanov.
\newblock Wave equations with strong damping in Hilbert spaces.
\newblock \emph{J. Differential Equations} \textbf{254} (2013), 3352--3368.

\bibitem{Jleli-Samet-Vetro=2021}
\newblock M. Jleli, B. Samet, C. Vetro.
\newblock A general nonexistence result for inhomogeneous semilinear wave equations with double damping and potential terms.
\newblock \emph{Chaos Solitons Fractals} \textbf{144} (2021), Paper No. 110673, 6 pp.

\bibitem{John=1979}
\newblock F. John.
\newblock Blow-up of solutions of nonlinear wave equations in three space dimensions.
\newblock \emph{Manuscripta Math.} \textbf{28} (1979), no. 1-3, 235--268.

\bibitem{Kainane-Kainane-Reissig=2020}
\newblock M. Kainane Mezadek, M. Kainane Mezadek, M. Reissig.
\newblock Semilinear wave models with friction and viscoelastic damping.
\newblock \emph{Math. Methods Appl. Sci.} \textbf{43} (2020), 3117--3147.

\bibitem{Kirane-Fino-Kerbal-Laadhari=2024}
\newblock M. Kirane, A. Z. Fino, S. Kerbal, A. Laadhari.
\newblock Non-existence of global weak solutions to semi-linear wave equations involving time-dependent structural damping terms.
\newblock \emph{Appl. Comput. Math.} \textbf{23} (2024), 110--129.

\bibitem{Michihisa=2021}
\newblock H. Michihisa.
\newblock Optimal leading term of solutions to wave equations with strong damping terms.
\newblock \emph{Hokkaido Math. J.} \textbf{50} (2021), 165--186.

\bibitem{Palmieri-Takamura=2019}
\newblock A. Palmieri, H. Takamura.
\newblock Blow-up for a weakly coupled system of semilinear damped wave equations in the scattering case with power nonlinearities.
\newblock \emph{Nonlinear Anal.} \textbf{187} (2019), 467--492.

\bibitem{Pham-Kainane-Reissig=2015}
\newblock D. T. Pham, M. Kainane Mezadek, M. Reissig.
\newblock Global existence for semi-linear structurally damped $\sigma$-evolution models.
\newblock \emph{J. Math. Anal. Appl.} \textbf{431} (2015), 569--596.

\bibitem{Schilling-Song-Vondracek=2012}
\newblock R. L. Schilling, R. Song, Z. Vondra\v{c}ek.
\newblock \emph{Bernstein Functions: Theory and Applications}, second ed.
\newblock De Gruyter Studies in Mathematics, vol. 37, Walter de Gruyter, Berlin, 2012.

\bibitem{Shibata=2000}
\newblock Y. Shibata.
\newblock On the rate of decay of solutions to linear viscoelastic equation.
\newblock \emph{Math. Methods Appl. Sci.} \textbf{23} (2000), 203--226.

\bibitem{Takamura-Wakasa=2011}
\newblock H. Takamura, K. Wakasa.
\newblock The sharp upper bound of the lifespan of solutions to critical semilinear wave equations in high dimensions.
\newblock \emph{J. Differential Equations} \textbf{251} (2011), 1157--1171.

\bibitem{Wakasa-Yordanov=2019}
\newblock K. Wakasa, B. Yordanov.
\newblock Blow-up of solutions to critical semilinear wave equations with variable coefficients.
\newblock \emph{J. Differential Equations} \textbf{266} (2019), 5360--5376.

\bibitem{Zhou=2001}
\newblock Y. Zhou.
\newblock Blow up of solutions to the Cauchy problem for nonlinear wave equations.
\newblock \emph{Chinese Ann. Math. Ser. B} \textbf{22} (2001), 275--280.

\end{thebibliography}
\end{document}